\documentclass[12pt]{amsart}
\usepackage{amsfonts}
\usepackage{amsfonts,latexsym,rawfonts,amsmath,amssymb,amsthm}
\usepackage[plainpages=false]{hyperref}
\usepackage{cite}

\usepackage{graphicx}

\RequirePackage{color}

\numberwithin{equation}{section}

\newcommand{\beq}{\begin{equation}}
\newcommand{\eeq}{\end{equation}}
\newcommand{\beqs}{\begin{eqnarray*}}
\newcommand{\eeqs}{\end{eqnarray*}}
\newcommand{\beqn}{\begin{eqnarray}}
\newcommand{\eeqn}{\end{eqnarray}}
\newcommand{\beqa}{\begin{array}}
\newcommand{\eeqa}{\end{array}}

\newtheorem{definition}{Definition}
\newtheorem{lemma}{Lemma}
\newtheorem{remark}{Remark}

\newtheorem{theorem}{Theorem}
\newtheorem{example}{Example}
\newtheorem{corollary}{Corollary}
\title  {iteration problem for several chaos in non-autonomous discrete system }
\thanks{*Corresponding author.}
\thanks{E-mail address: liubingwen2016@mail.zjxu.edu.cn (B. Liu)}
\begin{document}

%\address{Hongbo Zeng: Department of Mathematical Sciences, Tsinghua University,  China}

%\email{zenghongbo@csust.edu.cn }

%\date{}

\bibliographystyle{plain}

%\tableofcontents
\maketitle

\baselineskip=15.8pt
\parskip=3pt

\centerline {\bf   \small Hongbo Zeng$^{a}$, Chuangxia Huang$^{a}$, Bingwen Liu$^{b *}$ }
\centerline {$^{a}$School of Mathematics and Statistics, Changsha University of Science and Technology;}
\centerline {Hunan Provincial Key Laboratory of Mathematical Modeling and Analysis in Engineering,}
\centerline {Changsha 410114, Hunan, China}
\centerline {$^{b}$ College of Data Science, Jiaxing University, Jiaxing, 314001, Zhejiang,  China}

\vskip20pt

\noindent {\bf Abstract}:
%\begin{abstract}
In this paper we investigate the iteration problem for several chaos in non-autonomous discrete system $(X, f_{ 1,\infty} )$. Firstly, we prove that the Li-Yorke chaos of a non-autonomous discrete dynamical system is preserved under iterations when $f_{1,\infty}$ converges to $f$, which weakens the condition in the literature that $f_{1,\infty}$ uniformly converges to $f$. Besides, we prove that both DC2' and Kato's chaos of a non-autonomous discrete dynamical system are iteration invariants. Additionally, we give a sufficient condition for non-autonomous discrete dynamical system to be Li-Yorke chaos. Finally, we give an example to show that the DC3 of a non-autonomous discrete dynamical system is not inherited under iterations, which partly answers an open question proposed by Wu and Zhu(Chaos in a class of non-autonomous discrete systems, Appl.Math.Lett. 2013,26:431-436).

%\end{abstract}
 \vskip20pt
 \noindent{\bf Key Words:}  iteration invariant; distributional chaos; Li-Yorke chaos; non-autonomous discrete system.
 \vskip20pt

\vskip20pt

\baselineskip=15.8pt
\parskip=3pt

%\newpage

%\tableofcontents
\maketitle

\baselineskip=15.8pt
\parskip=3.0pt

%\begin{abstract}

%\end{abstract}

%\keywords{iteration invariant; distributional chaos; Li-York chaos}

%\begin{multicols}{2}
\section{Introduction}
\noindent  The discovery of chaos phenomenon and the initiation of chaos theory are one of the
greatest scientific discoveries in the last century. The study on chaos has become a major
project in nonlinear science, which has gotten a rapid development and rich achievements. There exists a chaotic phenomenon in almost all fields relating to dynamical progress and the chaos theory is a hot topic
in area of topological dynamics, which has had a great influence on modern
science including natural science and many humanities. Li and Yorke gave the definition of chaos first in 1975 (see \cite{a1}). Since then, different people from different
fields gave different definitions of chaos under their
understanding of the subject, such as Li-Yorke chaos, distributional chaos (i.e. Schweizer-Smital chaos, see \cite{a2}), Devaney chaos \cite{a3}, Kato's chaos \cite{a4}, etc. So it is an important and significant question to understand the relation among the various definitions. Nowadays, there are many results about that. Among them, distributional chaos is one of the important definition.  And more and more researchers give their attention to the properties of distributional chaos. Later, three mutually nonequivalent versions of distributional chaos of type 1-3 ($DC1-DC3$) were considered \cite{6f}. Recently, the author gave a the definitions of DC2' in \cite{a13}, witch is between DC1 and DC3.

\par
   As a natural extension of autonomous discrete dynamical systems (Abbrev. ADS), non-autonomous discrete dynamical systems (NDS) are an important part of topological
dynamical systems. Compared with classical dynamical systems (ADS), NDS can
describe various dynamical behaviors more flexibly and conveniently. Indeed, most of the natural phenomena are subjected to time-dependent external forces
and their modeling usually contain time-dependent parameters, modulation, and various
other effects. Meanwhile, the dynamics of non-autonomous discrete systems has became
an active research area, obtaining results on topological entropy, sensitivity, mixing properties, chaos, and other properties. Generally, non-autonomous systems have richer dynamics
than autonomous systems. Kolyada and Snoha\cite{a14} firstly studied the chaotic behavior of NDS. The research of the complex behavior of NDS is recently very intensive, see \cite{1a,a15,a17,a19,a191,a192,a193} and the references therein.
Assume that $\mathbb{N}=\{1,2,3,...\}$. Let $(X,d)$ be a compact metric
space and $f_n : X \rightarrow X$ be a sequence of continuous functions, where $n\in \mathbb{N}$. An non-autonomous discrete
dynamical systems (NDS) is a pair $(X, f_{ 1,\infty} )$ where $f_{ 1,\infty}= ( f_n )_{n=1}^\infty$. Given $i,n\in \mathbb{N}$.
Define the composition
$f_i^n:= f_{i+(n-1)} \circ \cdot\cdot\cdot \circ f_i$ ,
and usual $f_i^0=id_X$. In particular, if $f_n = f$ for all $n \in \mathbb{N}$, the pair $(X, f_{ 1,\infty} )$ is
just the classical discrete system $(X, f)$. The orbit of
a point $x$ in $X$ is the set
$Orb(X, f_{ 1,\infty}):=\{x,f_1^1(x),f_1^2(x),...,f_1^n(x),...\}$,
which can also be described by the difference equation $x_0 = x$ and $x_{n+1} = f_n (x_n )$.
For any $k \in \mathbb{N}$, the $k^{th}$-iterate of NDS $(X, f_{ 1,\infty} )$ is defined by
$f_{ 1,\infty}^{[k]}:=(f_{k(n-1)+1}^k)_{n=1}^\infty$.
Cnovas\cite{a19} studied the limit behaviour of sequence with the form $f_n \circ \cdot\cdot\cdot \circ f_1 (x)$,
$x \in [0,1]$, and wandered whether the simplicity (resp., chaoticity) of $f$ implies the
simplicity (resp., chaoticity) of $f_{ 1,\infty}$ , when $f_{ 1,\infty}$ converges uniformly to a map $f$. In general autonomous discrete system $(X, f )$, many authors (see, for instance, paper \cite{a21,a23} and the references therein) considered the iteration invariants of $f$ (for example, Li-Yorke chaos, distributional
chaos, large deviations theorem, shadowing property) and proved that $(X, f )$ is Li-Yorke chaotic, $DCi (i = 1,
2,2' )$ if and only if $f^N$
is too for any $N \in \mathbb{N}$,  that is, for all $N>0$, the compositional systems preserve chaos of the primary systems. But it is invalid for Devaney chaos and DC3. Recently,
the authors \cite{1a,a20} proved that the Li-Yorke chaos, DC1, DC2 and $(\mathcal{F}_1 ,\mathcal{F}_2 )$-Chaos
of NDS which converges uniformly are all inherited under iterations. Meanwhile, Wu
and Zhu\cite{1a} posed an open problem, they asked whether DC3 of $f_{ 1,\infty}$ is inheried under iterations, we also discuss and partly solve the problem in this paper. Besides, Shao et al.\cite{a17} proved that Li-Yorke $\delta$-chaos and distributional $\delta'$-chaos in a sequence are equivalent for NDS on
compact spaces, they also provided sufficient conditions for NDS to be distributionally chaotic.
Motivated by these, we would like to further investigate the iteration invariance of chaos in non-autonomous
discrete systems.

\par
   This paper is organized as follows. In Section 2, we will first state some preliminaries, definitions and some lemmas. The main conclusions will be given in Section 3.

\section{Preparations and lemmas}
In this section, we mainly give some different concepts of chaos for NDS and some lemmas. We always suppose that $(X, f_{ 1,\infty} )$  is a non-autonomous discrete system and that all the maps are continuous from $X$ to $X$ in the following context.

\begin{definition}
 If there exists an uncountable subset $S\subseteq X$ such that for any different points $x,y\in S$  we have
$$\liminf_{n\rightarrow\infty} d(f_1^{n}(x), f_1^{n}(y))=0,\ \limsup_{n\rightarrow\infty} d(f_1^{n}(x), f_1^{n}(y))>0,$$
then $f_{ 1,\infty}$ is called Li-Yorke chaos.
\end{definition}

\begin{definition}
Let
$$\Phi(f_{ 1,\infty},x,y,t)=\liminf_{n\rightarrow\infty}\frac{1}{n}\sharp \{0\leq i\leq n-1\mid
d(f_1^{i}(x), f_1^{i}(y))<t\},$$
$$\Phi^{*}(f_{ 1,\infty},x,y,t)=\limsup_{n\rightarrow\infty}\frac{1}{n}\sharp \{0\leq i\leq n-1\mid
d(f_1^{i}(x), f_1^{i}(y))<t\},$$ where $\sharp A$ denotes the cardinality of the set $A$.
If there exists an uncountable subset $S\subseteq X$ such that for any different points  $x,y\in S$ we have

(1) $\Phi(f_{ 1,\infty},x,y,\varepsilon)=0$ for some $\varepsilon >0$ and $\Phi^{*}(f_{ 1,\infty},x,y,t)=1$ for all $t>0$, then $f_{ 1,\infty}$ is called distributional chaos(or DC1).

(2) $\Phi(f_{ 1,\infty},x,y,\varepsilon)>0$ for some $\varepsilon >0$ and $\Phi^{*}(f_{ 1,\infty},x,y,t)=1$ for all $t>0$, then $f_{ 1,\infty}$ is called DC2.

(3)$\Phi(f_{ 1,\infty},x,y,\varepsilon)=0$ for some $\varepsilon >0$ and $\Phi^{*}(f_{ 1,\infty},x,y,t)>0$ for all $t>0$, then $f_{ 1,\infty}$ is called DC2'.

(4)$\Phi(f_{ 1,\infty},x,y,\varepsilon)<\Phi^{*}(f_{ 1,\infty},x,y,t)=1$ for all $t\in J$, where $J$ is some nondegenerate interval, then $f_{ 1,\infty}$ is called DC3.
\end{definition}

\begin{definition}
Let $\{p_k\}_{k=1}^{\infty}$ be a sequence of positive integers and
$$\Phi(f_{ 1,\infty},x,y,t,p_k)=\liminf_{n\rightarrow\infty}\frac{1}{n}\sharp \{0\leq i\leq n-1\mid
d(f_1^{p_i}(x), f_1^{p_i}(y))<t\},$$
$$\Phi^{*}(f_{ 1,\infty},x,y,t,p_k)=\limsup_{n\rightarrow\infty}\frac{1}{n}\sharp \{0\leq i\leq n-1\mid
d(f_1^{p_i}(x), f_1^{p_i}(y))<t\}.$$
If there exists an uncountable subset $S\subseteq X$ such that for any different points  $x,y\in S$ we have
$\Phi(f_{ 1,\infty},x,y,\varepsilon,p_k)=0$ for some $\varepsilon >0$ and $\Phi^{*}(f_{ 1,\infty},x,y,t,p_k)=1$ for all $t>0$, then $f_{ 1,\infty}$ is called distributively chaotic in a sequence. If there exist an uncountable subset $S\subseteq X$ and $\varepsilon >0$ such that for any different points  $x,y\in S$ we have
$\Phi(f_{ 1,\infty},x,y,\varepsilon,p_k)=0$ and $\Phi^{*}(f_{ 1,\infty},x,y,t,p_k)=1$ for all $t>0$, then $f_{ 1,\infty}$ is called uniformly distributively chaotic in a sequence.

\end{definition}

\begin{definition}
The non-autonomous system $(X, f_{ 1,\infty} )$  is said to be sensitive, if there is $\delta>0$ such that for any nonempty open set $U\subset X$, there exist $x,y\in U$ and $n\in \mathbb{N}$ such that $d(f_1^n(x),f_1^n(y))>\delta$. The non-autonomous system $(X, f_{ 1,\infty} )$  is said to be accessible, if for any $\varepsilon>0$ and any nonempty open set $U,V\subset X$, there exist $x\in U,y\in V$ and $n\in \mathbb{N}$ such that $d(f_1^n(x),f_1^n(y))<\varepsilon$. The non-autonomous system $(X, f_{ 1,\infty} )$ is said to be Kato's chaotic, if it is sensitive and accessible.

\end{definition}
Denote $N_{f_{1,\infty}}(U,\delta)=\{n\in\mathbb{N}: \text{there exist} \ x,y\in U \ \text{such that } d(f_1^n(x),f_1^n(x))<\delta \}$ for any nonempty open set $U$ of $X$.
\begin{definition}
A non-autonomous system $(X, f_{ 1,\infty} )$ is said to be finitely generated, if there exists a finite set $F$ of continuous self maps on $X$ such that each $f_i$ of $f_{ 1,\infty}$ belongs to $F$.
\end{definition}

For the convenience of the following proof, let $$\xi_n(f_{ 1,\infty},x,y,t):=\frac{1}{n}\sharp \{0\leq i\leq n-1\mid
d(f_1^{i}(x), f_1^{i}(y))<t\},$$
$$\delta_n(f_{ 1,\infty},x,y,t):=\frac{1}{n}\sharp \{0\leq i\leq n-1\mid
d(f_1^{i}(x), f_1^{i}(y))\ge t\},$$

\begin{lemma}\label{yinli1}
Let $\{n_k\}$ an increasing sequence of positive integers, then for any $n \in\mathbb{N}$, there exist a integers $r(0\le r<n)$, a subsequence $\{n_{k_j}\}$ of $\{n_k\}$ and an increasing sequence of positive integers $\{q_j\}$ such that $n_{k_j}=nq_j+r$.
\end{lemma}

\begin{lemma}\label{yinli2}\cite{1a}
Assume that non-autonomous discrete system $(X,f_{1,\infty})$ converges uniformly to a map $f$. Then for any $k\ge 2$, the sequence $(f_n^k)_{n=1}^{\infty}$ converges uniformly to $f^k$. Particularly, for any increasing sequence of positive integers $\{m_n\}_{n=1}^{\infty}$, we have $\{f_{m_n}^k\}_{n=1}^{\infty}$
converges uniformly to $f^k$.
\end{lemma}

\begin{lemma}\label{yinli3}
Assume that non-autonomous discrete system $(X,f_{1,\infty})$ converges uniformly to a map $f$. Then for any $\varepsilon >0$ and any $k\in \mathbb{N}$, there exist $\xi$ such that for any $n\in\mathbb{N}$ and any pair $x,y \in X$ with $d(x,y)<\xi$, $d(f_n^k(x),f_n^k(y))<\frac{\varepsilon}{2}$.
\end{lemma}
\begin{proof}
For any $\varepsilon >0$ and any $k$, on the one hand, according to Lemma \ref{yinli2}, there exist $N$ such that for any $n>N$ and any $x\in X$, we have
\begin{equation}\label{yishi}
d(f_n^k(x),f^k(x))<\frac{\varepsilon}{6} .
\end{equation}
In addition, $f$ is a continuous map because $(X,f_{1,\infty})$ converges uniformly to $f$, which implies $f^k$ is continuous for any $k$, then for $\varepsilon >0$ above, there exist $\xi_1$ such that for any pair $x,y \in X$ with $d(x,y)<\xi_1$,
\begin{equation}\label{ershi}
d(f^k(x),f^k(y))<\frac{\varepsilon}{6}.
\end{equation}
Combining (\ref{yishi}) and (\ref{ershi}), we have $$d(f_n^k(x),f_n^k(y))<d(f_n^k(x),f^k(x))+d(f^k(x),f^k(y))+d(f_n^k(y),f^k(y))<\frac{\varepsilon}{2}.$$
on the other hand, if $n\le N$, since $f_n^k$ is continuous,  there exist $\xi_2$ such that for any $n\le N$ and any pair $x,y \in X$ with $d(x,y)<\xi_2$, \begin{equation*}
d(f_n^k(x),f_n^k(y))<\frac{\varepsilon}{2}.
\end{equation*}
Finally, for any $\varepsilon >0$ and any $k$, we take $\xi=min\{\xi_1,\xi_2\}$, such that for any pair $x,y \in X$ with $d(x,y)<\xi$, $d(f_n^k(x),f_n^k(y))<\frac{\varepsilon}{2}$.
\end{proof}
\begin{remark}
The lemma \ref{yinli3} above is extension of the corollary 2.2 in \cite{1a}.
\end{remark}

\begin{lemma}(\cite{a49})\label{yinli20}
There is an uncountable subset $E$ in $\Sigma_2$ such that for any different points $s=s_0s_1...,t=t_0t_1...\in E$, we have $ s_n=t_n$ for infinitely many $n$ and $s_m\neq t_m$ for infinitely many $m$, where $\Sigma_2$ denotes the symbolic dynamical system.
\end{lemma}

\section{Main results}

\begin{theorem}
Assume that non-autonomous discrete system $(X,f_{1,\infty})$ converges to a map $f$. Then for any $k\in\mathbb{N}$,
 $f_{1,\infty}$ is Li-Yorke chaos if and only if $f_{1,\infty}^k$ is Li-Yorke chaos.
\end{theorem}
\begin{proof}
Sufficiency is obvious by definitions. We will prove the necessity in the following.
Since $f_{1,\infty}$ is Li-Yorke chaos, there exists an uncountable subset $S\subseteq X$ such that for any different points $x,y\in S$ we have
\begin{equation}\label{3.1}
\liminf_{n\rightarrow\infty}d(f_1^n(x),f_1^n(y))=0,
\end{equation}
and
\begin{equation}\label{3.2}
\limsup_{n\rightarrow\infty}d(f_1^n(x),f_1^n(y))>0.
\end{equation}
For (\ref{3.1}), there exist  $x_0\in X$ and an increasing sequence $\{n_i\}$ of positive integers such that
$$\lim_{i\rightarrow\infty}f_1^{n_i}(x)=x_0,$$
and
$$\lim_{i\rightarrow\infty}f_1^{n_i}(x)=x_0.$$
According to Lemma \ref{yinli1}, there exist  $r$ with $0\le r <k$ and $\{q_i\}\subseteq \{n_i\}$ such that
$$\lim_{i\rightarrow\infty}f_{r+1}^{kq_i}f_1^{r}(x)=x_0,$$
and
$$\lim_{i\rightarrow\infty}f_{r+1}^{kq_i}f_1^{r}(y)=x_0.$$
Since $(f_n)_{n=1}^\infty$ converges to $f$, then for the $x_0\in X$,
$$\lim_{i\rightarrow\infty}f_{kq_i+r+1}^{l}(x_0)=f^l(x_0),$$
where $l=k-r$, so we have
$$\lim_{i\rightarrow\infty}f_{1}^{k(q_i+1)}(x)=\lim_{i\rightarrow\infty}f_{kq_i+r+1}^{l}f_1^{kq_i+r}(x)=f^l(x_0),$$
and
$$\lim_{i\rightarrow\infty}f_{1}^{k(q_i+1)}(y)=\lim_{i\rightarrow\infty}f_{kq_i+r+1}^{l}f_1^{kq_i+r}(y)=f^l(x_0).$$
Therefore
\begin{equation}\label{3.3}
\liminf_{n\rightarrow\infty}d(f_1^{kn}(x),f_1^{kn}(y))=0.
\end{equation}
For (\ref{3.2}), there exist  $a,b\in X (a\neq b)$ and an increasing sequence $\{n_i'\}$ of positive integers such that
$$\lim_{i\rightarrow\infty}f_1^{n_i'}(x)=a,$$
and
$$\lim_{i\rightarrow\infty}f_1^{n_i'}(x)=b.$$
As the same way, there exist  $r$ with $0\le r <k$ and $\{p_i\}\subseteq \{n_i'\}$ such that
$$\lim_{i\rightarrow\infty}f_1^{kp_i+r}(x)=a,$$
and
$$\lim_{i\rightarrow\infty}f_1^{kp_i+r}(y)=b.$$
So there exist  $a', b'\in X$ and $\{p_{i_j}\}\subseteq \{p_i\}$ such that
$$\lim_{j\rightarrow\infty}f_1^{kp_{i_j}}(x)=a',$$
and
$$\lim_{j\rightarrow\infty}f_1^{kp_{i_j}}(y)=b'.$$
Since
$$a=\lim_{j\rightarrow\infty}f_{kq_{i_j}+1}^{r}f_1^{kq_{i_j}}(x)=\lim_{j\rightarrow\infty}f_{kq_{i_j}+1}^{r}(a')=f^r(a')$$
and
$$b=\lim_{j\rightarrow\infty}f_{kq_{i_j}+1}^{r}f_1^{kq_{i_j}}(y)=\lim_{j\rightarrow\infty}f_{kq_{i_j}+1}^{r}(b')=f^r(b'),$$
we have $f^r(a')\neq f^r(b')$, so $a'\neq b'$.

Therefore
\begin{equation}\label{3.4}
\limsup_{n\rightarrow\infty}d(f_1^{kn}(x),f_1^{kn}(y))>0.
\end{equation}

Summing up (\ref{3.3}) and (\ref{3.4}), it follows that $f_{1,\infty}^k$ is Li-Yorke chaos.

\end{proof}

\begin{remark}
The theorem above weaken the conditions of theorem 2.5 in \cite{1a}, where it is required that $(X,f_{1,\infty})$ converges uniformly to a map $f$.
\end{remark}

\begin{theorem}\label{ag}
 Assume that non-autonomous discrete system $(X,f_{1,\infty})$ converges uniformly to a map $f$. Then for any $N\in\mathbb{N}$, Then $f_{1,\infty}$ is DC2' if and only if $f_{1,\infty}^N$ is too.
\end{theorem}
\begin{proof}
Necessity. (1) Since $f_{1,\infty}$ is DC2', there exists an uncountable subset $S\subseteq X$ such that for any different points  $x,y\in S$ we have $\Phi(f_{1,\infty},x,y,\varepsilon)=0$ for some $\varepsilon >0$, then there exist an increasing sequence $\{n_k\}$ of positive integers such that for some $\varepsilon>0$,
\begin{equation}\label{3.5}
\lim_{k\rightarrow\infty}\frac{1}{n_k}\sharp \{0\leq i\leq n_k-1\mid
d(f_1^{i}(x), f_1^{i}(y))<\varepsilon\}=0,
\end{equation}
Put $m_k=[\frac{n_k}{N}]$, where $[\frac{n_k}{N}]$ denotes the integral part of $\frac{n_k}{N}$. Then for each $k$,
$$\xi_{m_k}(f_{1,\infty}^N,x,y,\varepsilon)\le \xi_{n_k}(f_{1,\infty},x,y,\varepsilon),$$
it follows from (\ref{3.5}) that for $k\rightarrow\infty$, $$\frac{1}{n_k}\xi_{m_k}(f_{1,\infty}^N,x,y,\varepsilon)\rightarrow 0,$$
 and further $$\frac{N}{n_k}\xi_{m_k}(f_{1,\infty}^N,x,y,\varepsilon)\rightarrow 0.$$ This gives for $k\rightarrow\infty$, $$\frac{1}{m_k}\xi_{m_k}(f_{1,\infty}^N,x,y,\varepsilon)\rightarrow 0.$$
 Therefore $$\Phi(f_{1,\infty}^N,x,y,\varepsilon)=0.$$

(2) For any $t>0$, by Lemma \ref{yinli3}, there exist $0<t_1<t$ such that for any $x,y\in X$ with $d(x,y)<t_1$, $d(f_n^i(x),f_n^i(y))<\frac{t}{2}$, for any $i\in \{1,...,N\}$ and any $n\in \mathbb{N}$.
Since $f_{1,\infty}$ is DC2', there exists an uncountable subset $S\subseteq X$ such that for any different points  $x,y\in S$ we have $\Phi^{*}(f_{1,\infty},x,y,t)>0$ for any $t>0$. Denote $\Phi^{*}(f_{1,\infty},x,y,t)=\beta$. Then the inequality above implies that there exists an increasing sequence $\{n_l\}$ of positive integers such that for any $l$,
\begin{equation}
\frac{1}{n_l}\sharp \{0\leq i< n_l\mid
d(f_1^i(x), f_1^i(y))<t_1\}>\frac{\beta}{2}.
\end{equation}
Then, for any $l$, one have $$\{0\leq i <n_l\mid
d(f_1^i(x), f_1^i(y))<t_1\}\subset \bigcup_{j=0}^{k-1}\{0\leq i <[\frac{n_l}{N}]+1\mid
d(f_1^{iN+j}(x), f_1^{iN+j}(y))<t_1\}.$$ And then this implies that there exist a subsequence $\{n_l'\}$ of $\{n_l\}$ and $0\le j< N-1$ such that for any $l$,
\begin{equation}
\begin{aligned}
&\sharp \{0\leq i< [\frac{n_l'}{N}]+1\mid
d(f_1^{iN+j}(x), f_1^{iN+j}(y))<t_1\}\\
\ge &\frac{1}{N}\sharp \{0\leq i< n_l'\mid
d(f_1^{i}(x), f_1^{i}(y))<t_1\}.
\end{aligned}
\end{equation}
By the choice of $t_1$, it follows that
\begin{equation}
\begin{aligned}
&\sharp \{0\leq i< [\frac{n_l'}{N}]+1\mid
d(f_1^{(i+1)N}(x), f_1^{(i+1)N}(y))<t\}\\
\ge &\sharp \{0\leq i< [\frac{n_l'}{N}]+1\mid
d(f_1^{iN+j}(x), f_1^{iN+j}(y))<t_1\}\\
\ge &\frac{1}{N}\sharp \{0\leq i< n_l'\mid
d(f_1^{i}(x), f_1^{i}(y))<t_1\}.
\end{aligned}
\end{equation}
Then
\begin{equation}
\begin{aligned}
&\Phi^{*}(f_{1,\infty}^N,x,y,t)\\
\ge &\limsup_{l\rightarrow\infty}\frac{\sharp \{0\leq i< [\frac{n_l'}{N}]+1\mid
d(f_1^{iN}(x), f_1^{iN}(y))<t\}}{[\frac{n_l'}{N}]+1}\\
\ge &\limsup_{l\rightarrow\infty}\frac{\sharp \{0\leq i< [\frac{n_l'}{N}]+1\mid
d(f_1^{(i+1)N}(x), f_1^{(i+1)N}(y))<t\}-1}{[\frac{n_l'}{N}]+1}\\
\ge &\limsup_{l\rightarrow\infty}\frac{\sharp \{0\leq i< n_l'\mid
d(f_1^{i}(x), f_1^{i}(y))<t_1\}-1}{N([\frac{n_l'}{N}]+1)}\\
> &\frac{\beta}{2}\\
> &0.
\end{aligned}
\end{equation}

Sufficiency. (1) Since $f_{1,\infty}^N$ is DC2', there exists an uncountable subset $S\subseteq X$ such that for any different points  $x,y\in S$ we have $\Phi(f_{1,\infty}^N,x,y,s)=0$ for some $s>0$, then there exist an increasing sequence $\{n_k\}$ of positive integers such that for some $\varepsilon>0$,
\begin{equation}\label{3.6}
\lim_{k\rightarrow\infty}\frac{1}{n_k}\sharp \{0\leq i\leq n_k-1\mid
d(f_1^{Ni}(x), f_1^{Ni}(y))<s\}=0,
\end{equation}
Applying Lemma \ref{yinli3}, it follows that there exists $0<p<s$ such that for any pair $x_1,y_1\in X$ with $d(x_1,y_1)<p$, $d(f_n^i(x_1),f_n^i(y_1))<s$ holds for any $1\le i \le N$ and any $n\in \mathbb{N}$. This means that for any $j$ with $d(f_1^{jN}(x),f_1^{jN}(y))\ge s$ and any $1\le i \le N$, $d(f_1^{jN-i}(x),f_1^{jN-i}(y))\ge p$. Then
\begin{equation}\label{3.7}
N(\delta_{n_k}(f_{1,\infty}^N,x,y,s)-1)\le \delta_{N{n_k}}(f_{1,\infty},x,y,p).
\end{equation}
Put $m_k=Nn_k$.
Notice that
\begin{equation}\label{3.8}
\frac{1}{{n_k}}\delta_{n_k}(f_{1,\infty}^N,x,y,s)=1-\frac{1}{{n_k}}\xi_{n_k}(f_{1,\infty}^N,x,y,s),
\end{equation}
and
\begin{equation}\label{3.9}
\frac{1}{{m_k}}\delta_{m_k}(f_{1,\infty},x,y,p)=1-\frac{1}{{m_k}}\xi_{m_k}(f_{1,\infty},x,y,p).
\end{equation}
Combining (\ref{3.7}),(\ref{3.8}) and (\ref{3.9}),we have that
$$\frac{1}{{m_k}}\xi_{m_k}(f_{1,\infty},x,y,p)\le\frac{1}{{n_k}}\xi_{n_k}(f_{1,\infty}^N,x,y,s)+\frac{1}{{n_k}}.$$
 This, together with (\ref{3.6}), imply that as $k\rightarrow\infty$,
 $$\frac{1}{{m_k}}\xi_{m_k}(f_{1,\infty},x,y,p)\rightarrow0.$$
 This shows that
$$\Phi(f_{1,\infty},x,y,p)=0.$$

(2) For any $t>0$, since $f_{1,\infty}$ is uniformly convergence, there exist $0<t_1<t$ such that for any $x,y\in X$ with $d(x,y)<t_1$, $d(f_n^i(x),f_n^i(y))<t$, for any $i\in \{1,...,N\}$ any $n\in \mathbb{N}$. Since $f_{1,\infty}^N$ is DC2', there exists an uncountable subset $S\subseteq X$ such that for any different points  $x,y\in S$ we have  $\Phi^{*}(f_{1,\infty}^N,x,y,s_1)>0$ for any $s_1>0$. Denote $\Phi^{*}(f_{1,\infty}^N,x,y,s_1)=\beta$. Then there exists an increasing sequence $\{m_l\}$ such that for $l>0$,
\begin{equation}\label{111}
\frac{1}{m_l}\sharp \{0\leq i\leq m_l\mid
d(f_1^{iN}(x), f_1^{iN}(y))<t_1\}> \frac{\beta}{2}.
\end{equation}
The choice of $t_1$ means that
$$\sharp \{0\leq i< Nm_l\mid
d(f_1^{i}(x), f_1^{i}(y))<t\}\ge N\sharp \{0\leq i< m_l\mid
d(f_1^{iN}(x), f_1^{iN}(y))<t_1\}.$$
This, together with (\ref{111}), imply that
\begin{equation}
\begin{aligned}
&\Phi^{*}(f,x,y,t)\\
\ge &\limsup_{l\rightarrow\infty}\frac{\sharp \{0\leq i< Nm_l\mid
d(f_1^{i}(x), f_1^{i}(y))<t\}}{Nm_l}\\
\ge &\limsup_{l\rightarrow\infty}\frac{\sharp \{0\leq i< m_l\mid
d(f_1^{iN}(x), f_1^{iN}(y))<t_1\}}{m_l}\\
> &\frac{\beta}{2} \\
> &0.
\end{aligned}
\end{equation}

\end{proof}

\begin{remark}
Just as in papers \cite{a26,a192} for the condition that DC1 and DC2 are iteration invariants, we believe that the conditions of theorem above also can be replaced to that $f_{1,\infty}$ is equi-continuous in $X$ or that $f_{1,\infty}$ is an $N$-convergent.
\end{remark}

\begin{theorem}
 Assume that $f_{1,\infty}$ is finitely generated or converges uniformly to a map $f$. Then for any $k\in\mathbb{N}$, $f_{1,\infty}$ is Kato's chaos if and only if $f_{1,\infty}^k$ is Kato's chaos.
\end{theorem}
\begin{proof}
Sufficiency is obvious. We will prove the necessity in two steps.

 Step 1. Since  $f_{1,\infty}$ is Kato's chaos, there exists a $\delta>0$ such that for any nonempty open set $U$ of $X$, there exist $x,y\in U,n\geq1$, such that
 \begin{equation}\label{lian}
d(f_1^n(x),f_1^n(y))>\delta.
\end{equation}
If $f_{1,\infty}$ is finitely generated, since each $f_i$ of $f_{1,\infty}$ is continuous and $X$ is compact metric space, $f_m^j$ is uniformly continuous $(\forall j=1,2,...,k, \forall m\geq1)$. Therefore, for given $\delta$ above, there exists $\delta_{1a}>0$, such that when $d(u,v)<\delta_{1a}$, we have
\begin{equation*}
d(f_m^j(u),f_m^j(v))<\delta (\forall j=1,2,...,k, \forall m\geq1).
\end{equation*}
If $f_{1,\infty}$ converges uniformly to a map $f$, then by Lemma \ref{yinli3}, for given $\delta$ above, there exists $\delta_{1b}>0$, such that when $d(u,v)<\delta_{1b}$, we have
\begin{equation*}
d(f_m^j(u),f_m^j(v))<\delta (\forall j=1,2,...,k, \forall m\geq1).
\end{equation*}
Take $\delta_{1}=min\{\delta_{1a},\delta_{1b}\}$, then when $d(u,v)<\delta_1$, we have
\begin{equation}\label{lianxu}
d(f_m^j(u),f_m^j(v))<\delta (\forall j=1,2,...,k, \forall m\geq1).
\end{equation}
 Firstly, we will claim that there exists $n\in N_{f_{1,\infty}}(U,\delta)$ satisfying $n>k$. If not, then take a nonempty open set $U_1\subseteq U$ with $diam(U_1)<\delta_1$, on the one hand, by (\ref{lianxu}) we have that $\forall x,y\in U_1,n\leq k$, $d(f_1^n(x),f_1^n(x))<\delta$, on the other hand, there exist $x,y\in U_1,n_1\geq1$, such that $d(f_1^n(x),f_1^n(x))>\delta$, so $n_1>k$. Thus there is a $n\in N_{f_{1,\infty}}(U,\delta)$ satisfying $n>k$($n=n_1$). Now, take $j$ satisfying $n=kq+j$, where $q,j$ are positive integer and $1\leq j\leq k$. Secondly,  we will claim that there exists $n_2\geq1$, such that $d(f_1^{kn_2}(x),f_1^{kn_2}(y))>\delta_1/2$. If not, then for any $n_2\geq1$, we have $d(f_1^{kn_2}(x),f_1^{kn_2}(y))\leq \delta_1/2$, by (\ref{lianxu}) we have that  $d(f_1^{kn_2+j}(x),f_1^{kn_2+j}(y))\leq \delta$, therefore $d(f_1^n(x),f_1^n(y))\leq \delta$, which is a contradiction to (\ref{lian}).

 Step 2. For any $\varepsilon>0$ and nonempty open sets $U,V$ of $X$. As the same way, whether $f_{1,\infty}$ is finitely generated or converges uniformly to $f$, we have that for $\varepsilon>0$ above, there exists $\delta_2>0$, such that when $d(u,v)<\delta_2$, we have
\begin{equation}\label{lianxu1}
d(f_{m'}^{j'}(u),f_{m'}^{j'}(v))<\varepsilon (\forall j'=1,2,...,k, \forall {m'}\geq1).
\end{equation}
Since  $f_{1,\infty}$ is Kato's chaos, for $\delta_2$ above, there exist $x\in U, y\in V, n'\geq1$ such that $d(f_1^{n'}(x),f_1^{n'}(y))<\delta_2$. Take $j'$ satisfying $n'+j'=kq'$, where $q',j'$ are positive integer and $1\leq j'\leq k$. By (\ref{lianxu1}), $d(f_1^{n'+j'}(x),f_1^{n'+j'}(y))<\varepsilon$, that is $d(f_1^{kq'}(x),f_1^{kq'}(y))<\varepsilon$. Therefore, for any $\varepsilon>0$ and nonempty open sets $U,V$ of $X$, there exist $x\in U, y\in V, q' \geq1$ such that $d(f_1^{kq'}(x),f_1^{kq'}(y))<\varepsilon$.

\end{proof}

\begin{theorem}
Let $\{p_k\}_{k=1}^{\infty}$ be a sequence of positive integers, $\{A_i\}_{i=0}^{\infty}$ and $\{B_i\}_{i=0}^{\infty}$ be decreasing sequences of compact sets satisfying $$\bigcap_{i=0}^{\infty}A_i={a}, \bigcap_{i=0}^{\infty}B_i={b},$$ where $a\neq b$. If $\forall c=C_1C_2...$, where $C_k=A_k$ or $B_k$ for $k=0,1,2...$, there exists $x_c\in X$, such that $\forall k\geq1$, we have $f_1^{p_k}(x_c)\in C_k$, then $f_{ 1,\infty}$ is uniformly distributively chaotic in a sequence.
\end{theorem}

\begin{proof}
Take the set $E$ as in lemma \ref{yinli20}. Then by the assumptions, for any $s=s_0s_1...\in E$, there exists $x_s\in X$, such that for each $k\geq1$, $n!<k\leq (n+1)!$, we have
\[f_1^{p_k}(x_s)\in\begin{cases}
A_k, &  s_n=0, \\
B_k, &  s_n=1.
\end{cases}\]
Put $D=\{x_s\mid s\in E\}$. It is easy to see that if $s\neq t$, then $x_s\neq x_t$. Since $E$ is uncountable set, $D$ is uncountable set.
Let $x_s,y_t\in D$ with $x_s\neq y_t$, where $s=s_0s_1...,t=t_0t_1...\in E$. By Lemma \ref{yinli20}, there exist sequences of positive integers $n_i\rightarrow\infty$ and $m_i\rightarrow\infty$ satisfying $s_{n_i}=t_{n_i},s_{m_i}\neq t_{m_i}$ for infinitely many $i$. Next we will just prove that $x_s,y_t$ are uniformly distributively chaotic in a sequence. The whole proof is divided into two steps.

Step 1. For any $\delta>0$, we take $i$ large enough such that $\frac{1}{n_i}<\frac{\delta}{2}$, and by the property of $f_1^{p_k}(x)$, for $n_i!<k\leq (n_i+1)!$, we have $d(f_1^{p_k}(x), f_1^{p_k}(y))<\delta$. Further, we have \begin{equation*}
\begin{aligned}
&\frac{1}{(n_i+1)!}\sharp \{1\leq k\leq (n_i+1)!\mid
d(f_1^{p_k}(x), f_1^{p_k}(y))<\delta\}\\
\geq&\frac{(n_i+1)!-n_i!}{(n_i+1)!}\\
=&1-\frac{1}{n_i+1}\\
\rightarrow& 1 (i\rightarrow\infty).
\end{aligned}
\end{equation*}
Therefore $\Phi_{xy}^{*}(f_{ 1,\infty}, \delta, p_{i})=1$.

Step 2. Put $\varepsilon=\frac{d(a,b)}{2}$. Take $i$ large enough such that $\frac{1}{m_i}<\frac{d(a,b)}{4}$, so $d(f_1^{p_k}(x), f_1^{p_k}(y))>\varepsilon$ for $m_i!<k\leq (m_i+1)!$. Thus
\begin{equation*}
\begin{aligned}
&\frac{1}{(m_i+1)!}\sharp \{1\leq k\leq (m_i+1)!\mid
d(f_1^{p_k}(x), f_1^{p_k}(y))<\varepsilon\}\\
\leq &\frac{m_i!}{(m_i+1)!}\\
=& \frac{1}{m_i+1}\\
\rightarrow & 0 (i\rightarrow\infty).
\end{aligned}
\end{equation*}

Therefore $\Phi_{xy}(f_{ 1,\infty}, \varepsilon, p_{i})=0$.

The entire proof is complete.

\end{proof}

Notice that uniformly distributively chaotic in a sequence implies Li-Yorke chaos by definitions, so we have the following result by the above theorem.

\begin{corollary}
Let $\{p_k\}_{k=1}^{\infty}$ be a sequence of positive integers, $\{A_i\}_{i=0}^{\infty}$ and $\{B_i\}_{i=0}^{\infty}$ be decreasing sequences of compact sets satisfying $$\bigcap_{i=0}^{\infty}A_i={a}, \bigcap_{i=0}^{\infty}B_i={b},$$ where $a\neq b$. If $\forall c=C_1C_2...$, where $C_k=A_k$ or $B_k$ for $k=0,1,2...$, there exists $x_c\in X$, such that $\forall k\geq1$, we have $f_1^{p_k}(x_c)\in C_k$, then $f_{ 1,\infty}$ is Li-Yorke chaos.
\end{corollary}

Next, we will give an example to show that $f_{ 1,\infty}$ is DC3 but $f_{ 1,\infty}^N$ is not.
\begin{example}
From \cite{a25} we know that there is a homeomorphism map $F:I^2\rightarrow I^2$ which display distribution chaos.
Put

\[f_i=\begin{cases}
F^{i+1}, &  i\in[2n-1:n\in \mathbb{N}], \\
F^{-i}, &  i\in[2n:n\in \mathbb{N}].
\end{cases}\]
It is easy to see that $f_{2n-1}\circ\cdot\cdot\cdot\circ f_1=F^{2n}$ and $f_{2n}\circ\cdot\cdot\cdot\circ f_0=id$. It is easy to see that the non-autonomous system $(I^2,f_{1,\infty})$ is DC3, but its 2nd iterate $(I^2,f_{1,\infty}^2)$ is not DC3. Additionally, it can be verified that the non-autonomous system $(I^2,f_{1,\infty})$ is Li-Yorke chaos but not distribution chaos, which means that for the non-autonomous system, distribution chaos implies Li-Yorke chaos but the inverse is not true.
\end{example}
\begin{remark}
From the example above we can see that if the condition 'uniform convergengce' was removed, then DC3 isnot preserved under iteration. However, if the condition 'uniform convergengce' was placed by 'convergengce', then Li-Yorke chaos is preserved under iteration too. We do not know whether DC1(DC2,DC3) is preserved under iteration if $f_{ 1,\infty}$ just converges to a map $f$.
\end{remark}

\noindent\textbf{Question.} Assume that non-autonomous discrete system $(X,f_{1,\infty})$ uniformly to a map $f$. Then for any $N\in\mathbb{N}$, Then $f_{1,\infty}$ is DC1-3 if and only if $f_{1,\infty}^N$ is DC1-3?

%\section*{Acknowledgments}

%\bibliography{mybibfile}

\end{document}